\documentclass[a4paper,12pt]{article}
\usepackage[utf8]{inputenc}
\usepackage[english]{babel}
\usepackage{amsmath, amsfonts, amsthm, amssymb}
\usepackage{mathrsfs}
\usepackage{graphicx}
\usepackage{enumerate}
\usepackage{float}
\usepackage{hyperref}
\usepackage[margin=2.5cm]{geometry}
\usepackage{tikz-cd}
\usepackage[nameinlink]{cleveref}

\usetikzlibrary{decorations.pathmorphing}

	\newcommand{\bb}[1]{\mathbb{#1}}
	\newcommand{\ceil}[1]{\lceil#1\rceil}

	\newcommand{\inv}[1]{#1^{-1}}

	\newcommand{\smpt}[1]{\setminus\{#1\}}

	\newcommand{\vol}{\mathrm{vol}}
	
	\newtheorem{thm}{Theorem}[section]
	\newtheorem*{thm*}{Theorem}

	\newtheorem{lem}[thm]{Lemma}
	
	\newtheorem{prop}[thm]{Proposition}
	
	\theoremstyle{definition}
	\newtheorem{defn}[thm]{Definition}
	
	\theoremstyle{remark}
	\newtheorem*{rem}{Remark}
	
	\theoremstyle{remark}
	
	\theoremstyle{remark}
	
	\title{The density of visible points in the Ammann-Beenker point set}
	\author{Gustav Hammarhjelm}
	
\begin{document}

	\maketitle
	
	\begin{abstract}
	The relative density of visible points of the integer lattice $\bb{Z}^d$ is known to be $1/\zeta(d)$ for $d\geq 2$, where $\zeta$ is Riemann's zeta function. In this paper we prove that the relative density of visible points in the Ammann-Beenker point set is given by $2(\sqrt{2}-1)/\zeta_K(2)$, where $\zeta_K$ is Dedekind's zeta function over  $K=\bb{Q}(\sqrt{2})$. 
	\end{abstract}
	
	\section{Introduction}
	\label{secIntro}
	
	A locally finite point set $\mathcal{P}\subset \bb{R}^d$ has an \emph{asymptotic density} (or simply \emph{density}) $\theta(\mathcal{P})$ if
	\[\lim_{R\to\infty}\frac{\#(\mathcal{P}\cap RD)}{\vol(RD)}=\theta(\mathcal{P})\]
	holds for all Jordan measurable $D\subset \bb{R}^d$. The density of a set can be interpreted as the asymptotic number of elements per unit volume. For instance, for a lattice $\mathcal{L}\subset \bb{R}^d$ we have $\theta(\mathcal{L})=\frac{1}{\vol(\bb{R}^d/\mathcal{L})}$. Let
	$\widehat{\mathcal{P}}=\{x\in \mathcal{P}\mid tx\notin \mathcal{P}, \forall t\in (0,1)\}$ denote the subset of the \emph{visible} points of $\mathcal{P}$. If $\mathcal{P}$ is a regular cut-and-project set (see \Cref{defnCPS} below) then it is known that $\theta(\mathcal{P})$ exists. In \cite[Theorem 1]{marklof2014visibility}, J. Marklof and A. Strömbergsson proved that $\theta(\widehat{\mathcal{P}})$ also exists and that $0<\theta(\widehat{\mathcal{P}})\leq\theta(\mathcal{P})$ if $\theta(\mathcal{P})>0$. In particular, for such $\mathcal{P}$ the \textit{relative density of visible points} $\kappa_\mathcal{P}:=\frac{\theta(\widehat{\mathcal{P}})}{\theta(\mathcal{P})}$ exists, but is not known explicitly in most cases.
	
	For $d\geq 2$ we have $\widehat{\bb{Z}^d}=\{(n_1,\ldots,n_d)\in \bb{Z}^d\mid \gcd(n_1,\ldots,n_d)=1\}$ and  $\theta(\widehat{\bb{Z}^d})=1/\zeta(d)$ gives the probability that $d$ random integers share no common factor. This can be derived in several ways, see for instance \cite{nymann1972probability}; we sketch another proof in \Cref{secDensZn} below. More generally, $\theta(\widehat{\mathcal{L}})=\frac{1}{\vol(\bb{R}^d/\mathcal{L})\zeta(d)}$ for a lattice $\mathcal{L}\subset \bb{R}^d$, see e.g. \cite[Prop. 6]{baake2000diffraction}. 
	
	A well-known point set, which can be realised both as the vertices of a substitution tiling and as a cut-and-project set, is the Ammann-Beenker point set. The goal of this paper is to prove that the relative density of visible points in the Ammann-Beenker point set is $2(\sqrt{2}-1)/\zeta_K(2)$. This density was computed by B. Sing in the presentation \cite{singppt1}, but he has not published a proof of this result. 
	
	\section{The density of the visible points of $\bb{Z}^d$}
	\label{secDensZn}
	
	In this section we show that $\theta(\widehat{\bb{Z}^d})=1/\zeta(d)$. We shall see that a lot of inspiration can be drawn from this example when calculating the density of the visible points in the Ammann-Beenker point set.
	
	Fix $R>0$, a Jordan measurable $D\subset \bb{R}^d$ and let $\bb{P}\subset \bb{Z}_{>0}$ denote the set of prime numbers. For each \emph{invisible} point  $n\in \bb{Z}^d\setminus \widehat{\bb{Z}^d}$, there is $p\in\bb{P}$ such that $\frac{n}{p}\in\bb{Z}^d$. Setting $\bb{Z}^d_*=\bb{Z}^d\smpt{(0,\ldots,0)}$ there are only finitely many $p_1,\ldots,p_n\in\bb{P}$ such that $p_i\bb{Z}^d_*\cap RD\neq \emptyset$. By inclusion-exclusion counting we have
	\begin{align*}
	\#(\widehat{\bb{Z}^d}\cap RD)&=\#\left((\bb{Z}^d_*\cap RD)\setminus \bigcup_{p\in\bb{P}}(p\bb{Z}^d_*\cap RD)\right)=\#\left((\bb{Z}^d_*\cap RD)\setminus \bigcup_{i=1}^n(p_i\bb{Z}^d_*\cap RD)\right)\\
	&=\#(\bb{Z}^d_*\cap RD)+\sum_{k=1}^{m}(-1)^{k}\left(\sum_{1\leq i_1<\ldots <i_k\leq m}\#(p_{i_1}\bb{Z}^d_*\cap\cdots\cap p_{i_k}\bb{Z}^d_*\cap RD)\right).
	\end{align*}
	The last sum can be rewritten to
	\[\sum_{n\in\bb{Z}_{>0}}\mu(n)\cdot\#(n\bb{Z}^d_*\cap RD),\]
	where $\mu$ is the Möbius function. Hence
	\[\frac{\#(\widehat{\bb{Z}^d}\cap RD)}{\vol(RD)}=\sum_{n\in\bb{Z}_{>0}}\frac{\mu(n)\cdot\#(n\bb{Z}^d_*\cap RD)}{\vol(RD)}=\sum_{n\in\bb{Z}_{>0}}\frac{\mu(n)}{n^d}\frac{\#(\bb{Z}^d_*\cap \inv{n}RD)}{\vol(\inv{n}RD)}.\]
	Letting $R\to\infty$, switching order of limit and summation (for instance justified by finding a constant $C$ depending on $D$ such that $\#(\bb{Z}_*^d\cap RD)\leq C\vol(RD)$ for all $R$), using $\theta(\bb{Z}^d_*)=1$ and $1/\zeta(s)=\sum_{n\in\bb{Z}_{>0}}\frac{\mu(n)}{n^s}$ for $s>1$, we find that
	\[\theta(\widehat{\bb{Z}^d})=\lim_{R\to\infty}\frac{\#(\widehat{\bb{Z}^d}\cap RD)}{\vol(RD)}=1/\zeta(d).\]
	
	\section{Cut-and-project sets and the Ammann-Beenker point set}
	\label{secCPS}
	The Ammann-Beenker point set can be obtained as the vertices of the Ammann-Beenker tiling, a substitution tiling of the plane using a square and a rhombus as tiles, see e.g. \cite[Chapter 6.1]{baake2013aperiodic}. In this paper however, the Ammann-Beenker set is realised as a \emph{cut-and-project set}, a certain type of point set which we will now define. Cut-and-project sets are sometimes called (Euclidean) model sets. We will use the same notation and terminology for cut-and-project sets as in \cite[Sec. 1.2]{marklof2014free}. For an introduction to cut-and-project sets, see e.g. \cite[Ch. 7.2]{baake2013aperiodic}.
	 
	If $\bb{R}^n=\bb{R}^d\times \bb{R}^m$, let 
	\begin{alignat*}{2}
		\pi:& ~\bb{R}^n\longrightarrow \bb{R}^d  & \pi_{\mathrm{int}}: & ~ \bb{R}^n\longrightarrow \bb{R}^m \\
		& (x_1,\ldots,x_n)\longmapsto (x_1,\ldots,x_d)\hspace{1cm}& & (x_1,\ldots,x_n)\longmapsto (x_{d+1},\ldots, x_n)
	\end{alignat*}
	denote the natural projections.
	
	\begin{defn}
	\label{defnCPS}
	Let $\mathcal{L}\subset \bb{R}^n$ be a lattice and $\mathcal{W}\subset \overline{\pi_{\mathrm{int}}(\mathcal{L})}$ be a set. Then the \emph{cut-and-project} set of $\mathcal{L}$ and $\mathcal{W}$ is given by
	$\mathcal{P}(\mathcal{W},\mathcal{L})=\{\pi(y)\mid y\in \mathcal{L},\pi_{\mathrm{int}}(y)\in\mathcal{W}\}$.
	\end{defn}

	If $\partial W$ has measure zero with respect to any Haar measure on $\overline{\pi_{\mathrm{int}}(\mathcal{L})}$ we say that $\mathcal{P}(\mathcal{W},\mathcal{L})$ is \textit{regular}.
	If the interior of $\mathcal{W}$ (the \emph{window}) is non-empty, $\mathcal{P}(\mathcal{W},\mathcal{L})$ is relatively dense and if $\mathcal{W}$ is bounded, $\mathcal{P}(\mathcal{W},\mathcal{L})$ is uniformly discrete (cf. \cite[Prop. 3.1]{marklof2014free}). To realise the Ammann-Beenker point set in this way, let $K$ be the number field $\bb{Q}(\sqrt{2})$, with algebraic conjugation $x\mapsto \overline{x}$ (we will also write $\overline{x}=(\overline{x_1},\ldots,\overline{x_n})$ for $x=(x_1,\ldots,x_n)\in K^n$) and norm $N(x)=x\overline{x}$. The ring of integers $\mathcal{O}_K=\bb{Z}[\sqrt{2}]$ of $K$ is a Euclidean domain with fundamental unit $\lambda:=1+\sqrt{2}$. With $\zeta:=e^{\tfrac{\pi i}{4}}$ and $\star:K\longrightarrow K$, $x\mapsto x^\star$ the automorphism generated by $\zeta \mapsto \zeta^3$, the Ammann-Beenker point set is in \cite[Example 7.7]{baake2013aperiodic} realised as
	\[\{x=x_1+x_2\zeta\mid x_1,x_2\in\mathcal{O}_K,x^\star\in W_8\},\]
	where $W_8\subset \bb{C}$ is the regular octagon of side length $1$ centered at the origin, with sides perpendicular to the coordinate axes. 
	
	Let \[\mathcal{L}=\{(x,\overline{x})\mid x=(x_1,x_2)\in \mathcal{O}_K^2\}\subset \bb{R}^4\] be the Minkowski embedding of $\mathcal{O}_K^2$ and let \[\widetilde{\mathcal{L}}=\{(x,\overline{x})\in \mathcal{L}\mid (x_1-x_2)/\sqrt{2}\in \mathcal{O}_K\}.\] Then, after a straight-forward translation it is seen that the Ammann-Beenker point set $\mathcal{A}$ can be realised in $\bb{R}^2$ as
	$\mathcal{A}=\frac{1}{\sqrt{2}}\mathcal{P}(\mathcal{W}_\mathcal{A},\widetilde{\mathcal{L}})$,
	where $\mathcal{W}_\mathcal{A}:=\sqrt{2}W_8$, i.e. $\mathcal{A}$ is the scaling of a cut-and-project set according to \Cref{defnCPS}. 
	
	\section{The density of visible points of $\mathcal{A}$}
	\label{secDensAB}
	
	All notation used in this section is defined in and taken from \Cref{secCPS}. Since, for any $\mathcal{P}\subset \bb{R}^d$ whose density exists, and any $c>0$ it holds that $\theta(c\mathcal{P})=c^{-d} \theta(\mathcal{P})$ and $c\widehat{\mathcal{P}}=\widehat{c\mathcal{P}}$, finding $\theta(\widehat{\mathcal{A}'})$ with  $\mathcal{A}':=\sqrt{2}\mathcal{A}=\mathcal{P}(\mathcal{W}_\mathcal{A},\widetilde{\mathcal{L}})$ will give the value of $\theta(\widehat{\mathcal{A}})$. As a first step, in \Cref{subsecDensAB'}, the asymptotic density of the visible points of the simpler set $\mathcal{B}=\mathcal{P}(\mathcal{W}_\mathcal{A},\mathcal{L})=\{x\in\mathcal{O}_K^2\mid \overline{x}\in\mathcal{W}_\mathcal{A}\}\subset\mathcal{O}_K^2$ will be calculated. In \Cref{subsecDensAB} this result will be used to obtain $\theta(\widehat{\mathcal{A}})$.
	
	\subsection{The density of visible points of $\mathcal{B}$}
	\label{subsecDensAB'}
	
	The following general counting formula for bounded subsets of visible points of a point set $\mathcal{P}$ will be needed. Let $\mathcal{P}_*=\mathcal{P}\setminus\{(0,\ldots,0)\}$.
	
	\begin{prop}
	\label{propInclExcl}
	Let $\mathcal{P}\subset \bb{R}^d$ be locally finite and fix a set $C\subset\bb{R}_{>1}$ such that for each $x\in \mathcal{P}\setminus\widehat{\mathcal{P}}$ there exists $c\in C$ with $x/c\in \mathcal{P}$. Let $R>0$ and a bounded set $D\subset \bb{R}^d$ be given. Then 
	\[\#(\widehat{\mathcal{P}}\cap RD)=\sum_{\substack{F\subset C\\\#F<\infty}}(-1)^{\#F}\#\left(\left(\mathcal{P}_*\cap \bigcap_{c\in F}c\mathcal{P}_*\right)\cap RD\right).\]
	\end{prop}
	
	\begin{proof}
		The set $C_R:=\{c\in C\mid \mathcal{P}_*\cap c\mathcal{P}_*\cap RD\neq \emptyset\}$ is finite. Indeed, suppose this is not true and pick distinct $c_1,c_2,\ldots\in C_R$ and corresponding $x_i\in \mathcal{P}_*\cap c_i\mathcal{P}_*\cap RD$. Since $\mathcal{P}$ is locally finite, the sequence $x_1,x_2,\ldots$ contains only finitely many distinct elements. Thus, a subsequence $x_{k_1},x_{k_2},\ldots$ which is constant can be extracted, so that $x_{k_i}/c_{k_i}\in \mathcal{P}_*\cap \frac{RD}{c_{k_i}}\subset \mathcal{P}_*\cap RD$ are all distinct, contradiction to $\mathcal{P}$ being locally finite. Thus, we can write $C_R=\{c_1,\ldots,c_n\}$ for some $c_1,\ldots,c_n\in C$. Then
		\begin{align*}
		&\#(\widehat{\mathcal{P}}\cap RD)=\#\left((\mathcal{P}_*\cap RD)\setminus\bigcup_{c\in C}(\mathcal{P}_*\cap c\mathcal{P}_*\cap RD)\right)\\
		&=\#(\mathcal{P}_*\cap RD)-\#\left(\bigcup_{i=1}^n(\mathcal{P}_*\cap c_i\mathcal{P}_*\cap RD)\right),
		\end{align*}
		from which the result follows from the inclusion-exclusion counting formula for finite unions of finite sets.
	\end{proof}	
	
	A set $C$ as in \Cref{propInclExcl} for $\mathcal{B}$ will be needed, and to this end a visibility condition for the elements of $\mathcal{B}$ is required. Given $x_1,x_2\in \mathcal{O}_K$, let $\gcd(x_1,x_2)$ be a fixed generator of the ideal generated by $x_1,x_2$ and write $\gcd(x_1,x_2)=1$ when $x_1,x_2$ are relatively prime. In the following proposition a visibility condition of the complex realisation of the Ammann-Beenker point set given in \cite[p. 477]{baake2014radial} is adapted to our situation.
	
	\begin{prop}
	\label{propVisCond}
	The visible points of $\mathcal{B}$ are given by \[\widehat{\mathcal{B}}=\{x=(x_1,x_2)\in \mathcal{B}\mid \gcd(x_1,x_2)=1,\lambda \overline{x}\notin \mathcal{W}_\mathcal{A}\}.\]
	\end{prop}
	
	\begin{proof}
	First the  necessity of the visibility conditions is established. Take $x=(x_1,x_2)$ and suppose that $\gcd(x_1,x_2)\neq 1$ so that there exists $c\in\mathcal{O}_K$ with $|N(c)|>1$ and $c\mid x_1,x_2$. Scaling $c$ by units we may assume that $1<c<\lambda$. Suppose first that $|N(c)|=|\overline{c}|c\geq 3$, which implies $|\overline{c}|>1$. By noting that $\mathcal{W}_\mathcal{A}$ is star-shaped with respect to the origin and $\mathcal{W}_\mathcal{A}=-\mathcal{W}_\mathcal{A}$ it follows that $x/c\in\mathcal{B}$, so $x$ is invisible. If $|N(c)|=2$, then each prime factor of $c$ must divide $2=\sqrt{2}\cdot \sqrt{2}$, so it can be assumed that $c=\sqrt{2}$  and hence $x$ is occluded by $x/\sqrt{2}$. If $\lambda \overline{x}\in\mathcal{W}_\mathcal{A}$ it follows immediately that $x/\lambda\in\mathcal{B}$.
			
	We now turn to the sufficiency of the visibility conditions. Take $x=(x_1,x_2)\in\mathcal{B}\setminus\widehat{\mathcal{B}}$ and $c>1$ such that $x/c\in\mathcal{B}$. As $\mathcal{B}$ is uniformly discrete, we may assume that $y:=x/c\in \widehat{\mathcal{B}}$. This implies, by necessity above, that $\gcd(y_1,y_2)=1$. Now, since $x_i=cy_i$ it follows that $c\in K$. Write $c=a/b$ with $a,b\in\mathcal{O}_K$ relatively prime. If $b$ is not a unit, $\gcd(y_1,y_2)=1$ is contradicted, hence $c\in\mathcal{O}_K$.

	If $|N(c)|\neq 1$ then $\gcd(x_1,x_2)\neq1$. Otherwise, $c>1$ is a unit, i.e. $c=\lambda^k$ for some integer $k>0$. Thus $\frac{\overline{x}}{\overline{c}}=\frac{\overline{x}}{\overline{\lambda}^k}\in\mathcal{W}_\mathcal{A}$. Since $\frac{1}{\overline{\lambda}}=-\lambda$ we get $(-\lambda)^k\overline{x}\in\mathcal{W}_\mathcal{A}$ and thus also $\lambda \overline{x}\in\mathcal{W}_\mathcal{A}$ as $\mathcal{W}_\mathcal{A}$ is star-shaped with respect to the origin and $-\mathcal{W}_\mathcal{A}=\mathcal{W}_\mathcal{A}$. This establishes sufficiency of the visibility conditions.	
	\end{proof}
	
	\begin{rem}
	Note that the proof works just as well for more general windows, that is, $\widehat{\mathcal{P}(\mathcal{W},\mathcal{L})}=\{x\in \mathcal{P}(\mathcal{W},\mathcal{L})\mid \gcd(x_1,x_2)=1,\lambda \overline{x}\notin \mathcal{W}\}$ if $\mathcal{W}\subset\bb{R}^2$ is bounded with non-empty interior, star-shaped with respect to the origin and $-\mathcal{W}=\mathcal{W}$.
	\end{rem}
	
	Let now \[\bb{P}=\{\pi\in\mathcal{O}_K\mid \pi ~\mathrm{prime}, 1<\pi<\lambda\}~\text{and}~  C=\bb{P}\cup \{ \lambda\}\] so that $\bb{P}$ is a set that contains precisely one associate of every prime of $\mathcal{O}_K$.
	Then we have the following proposition.
	
	\begin{prop}
	\label{propOcclSet}
	For each $x\in\mathcal{B}\setminus\widehat{\mathcal{B}}$ there is $c\in C$ such that $x/c\in \mathcal{B}$.
	\end{prop}
	
	\begin{proof} Fix $x\in\mathcal{B}\setminus\widehat{\mathcal{B}}$.
	As seen in the proof of \Cref{propVisCond} there is $c\in \mathcal{O}_K$, $c>1$, such that $x/c\in \mathcal{B}$. If $c$ is not a unit, fix $\pi\in \bb{P}$ so that $\pi\mid c$. It can be verified by hand that $\{(x,\overline{x})\mid x\in\mathcal{O}_K\}\cap ((1,\lambda)\times (-1,1))=\emptyset$, hence $|\pi|>1$ and $x/\pi\in \mathcal{B}$. If $c$ is a unit, $x/\lambda\in\mathcal{B}$ is immediate.
	\end{proof}
	
	Given a finite set $F\subset \mathcal{O}_K$ let $I_F$ be the (principal) ideal generated by the elements of $F$ if $F\neq \emptyset$ and $I_F=\mathcal{O}_K$ otherwise. Let $\ell_F$ denote a fixed \emph{least common multiple} of $F$, that is, a generator of the ideal $\bigcap_{c\in F}c\mathcal{O}_K$. Let also $m_F=\min\{1,\min_{c\in F}|\overline{c}|\}$ and $\mathcal{L}_F=\{(\ell_Fx,\overline{\ell_Fx})\mid x\in \mathcal{O}_K^2\}$. Write $I\triangleleft \mathcal{O}_K$ when $I\subset \mathcal{O}_K$ is an ideal and define the \emph{absolute norm} $N(I)$ of $I$ by $|N(x)|$, where $x$ is any generator of $I$. Recall \emph{Dedekind's zeta function} $\zeta_K(s)=\sum_{I\triangleleft \mathcal{O}_K}\frac{1}{N(I)^s}$ for $s\in \bb{C}$ with $\mathrm{Re}(s)>1$.
	
	Given a finite set $F\subset C$ it is verified that $\mathcal{B}_*\cap \bigcap_{c\in F}c\mathcal{B}_*=\mathcal{P}(m_F\mathcal{W}_\mathcal{A},\mathcal{L}_F)\smpt{0}$. For any $R>0$ and bounded $D\subset \bb{R}^2$, Propositions \ref{propInclExcl}, \ref{propOcclSet} imply that
	\begin{equation}
	\label{eqA'1}
	\#(\widehat{\mathcal{B}}\cap RD)=\sum_{\substack{F\subset C\\\#F<\infty}}(-1)^{\#F}\#\left((\mathcal{P}(m_F\mathcal{W}_\mathcal{A},\mathcal{L}_F)\smpt{0})\cap RD\right).
	\end{equation}	
	Since $\ell_F\mathcal{O}_K^2\subset\pi_{\mathrm{int}}(\mathcal{L}_F)\subset\bb{R}^2$ is dense we have \[\theta(\mathcal{P}(m_F\mathcal{W}_\mathcal{A},\mathcal{L}_F)\smpt{0})=\frac{\vol(m_F\mathcal{W}_\mathcal{A})}{\vol(\bb{R}^4/\mathcal{L}_F)}\] from \cite[Prop. 3.2]{marklof2014free}. Dividing \eqref{eqA'1} by $\vol(RD)$, letting $R\to\infty$ and switching order of limit and summation (to be justified in \Cref{propInterchange} below) we find that \[\theta(\widehat{\mathcal{B}})=\sum_{\substack{F\subset C\\\#F<\infty}}(-1)^{\#F}\frac{\vol(m_F\mathcal{W}_\mathcal{A})}{\vol(\bb{R}^4/\mathcal{L}_F)}=\sum_{\substack{F\subset C\\\#F<\infty}}(-1)^{\#F}\frac{m_F^2(1+\sqrt{2})}{2N(\ell_F)^2},\]
	since $\vol(\mathcal{W}_\mathcal{A})=4(1+\sqrt{2})$ and $\vol(\bb{R}^4/\mathcal{L}_F)=8N(\ell_F)^2$.
	The value of the right hand sum will be shown to be $1/\zeta_K(2)$ in \Cref{thmAsDensA'} below. The following lemma gives a bound on the number of points in the intersection of a lattice and a box in terms of the volume of the box, provided that the box is "not too thin".
	
	\begin{lem}
	\label{lemLatticeBoxBound}
	Let $\mathcal{L}\subset \bb{R}^d$ be a lattice and let $c>0$ be given. For any $a_i,b_i\in\bb{R}$ with $b_i-a_i>c$ set $B=\prod_{i=1}^{d}[a_i,b_i]$. Then there is a constant $L$ depending only on $\mathcal{L}$ and $c$ such that $\#(B\cap\mathcal{L})\leq L\vol(B)$.
	\end{lem}
			
	\begin{proof}
	Let $n_i=\ceil{\frac{b_i-a_i}{c}}\in\bb{Z}_+$. Then
	$\frac{b_i-a_i}{c}\leq n_i< \frac{b_i-a_i}{c}+1=\frac{b_i-a_i+c}{c}<\frac{2(b_i-a_i)}{c}.$
	Hence, with $n=\prod_{i=1}^dn_i$ it follows that $n\leq \frac{2^d\vol(B)}{c^d}$. From $b_i\leq a_i+cn_i$ also
	$B\subset\prod_{i=1}^d[a_i,a_i+cn_i].$
	Let $a=(a_1,\ldots,a_d)$ and
	consider $-a+\prod_{i=1}^d[a_i,a_i+cn_i]=\prod_{i=1}^d[0,cn_i]$. We have
	$\prod_{i=1}^d[0,cn_i]=\bigcup_{m\in\bb{N}^d,0\leq m_i<n_i}(mc+[0,c]^d).$
	Hence, $B\subset \bigcup_{m\in\bb{N}^d,0\leq m_i<n_i}(a+mc+[0,c]^d)=:B'$. Find now $D>0$ depending on $\mathcal{L}$ and $c$ such that $\sup_{t\in\bb{R}^d}\#(\mathcal{L}\cap (t+[0,c]^d))=D$. Hence
	$\#(B\cap \mathcal{L})\leq \#(B'\cap\mathcal{L})\leq n D\leq \frac{2^dD}{c^d}\vol(B)$, so one can take $L=\frac{2^dD}{c^d}$.
	\end{proof}
	
	The following bound will be crucial in the justification of interchanging limit and summation in \eqref{eqA'1} after division by $\vol(RD)$.
	
	\begin{lem}
		\label{lem-Estimate1}
		Let $D\subset \bb{R}^2$ be Jordan measurable.
		Then there is a constant $\widetilde{L}>0$ depending only on $D$ such that for every $R>0$ and $F\subset C$ with $\#F<\infty$,
		\[\#((\mathcal{P}(m_F\mathcal{W}_\mathcal{A},\mathcal{L}_F)\cap RD)\smpt{0})\leq \frac{\widetilde{L}R^2}{N(\ell_F)^2}.\]
		\end{lem} 
		
		\begin{proof}
		By definition \[\#((\mathcal{P}(m_F\mathcal{W}_\mathcal{A},\mathcal{L}_F)\smpt{0})\cap RD)=\#(\{x\in \ell_F\mathcal{O}_K^2\mid\overline{x}\in m_F\mathcal{W}_\mathcal{A}\}\smpt{0})\cap RD).\] Note that this number is independent of the choice of $\ell_F$. There is a bijection \[(\{x\in \ell_F\mathcal{O}_K^2\mid \overline{x}\in m_F\mathcal{W}_\mathcal{A}\}\smpt{0})\cap RD\longrightarrow(\{x\in\mathcal{O}_K^2\mid \overline{x}\in \tfrac{m_F\mathcal{W}_\mathcal{A}}{|\overline{\ell_F}|}\}\smpt{0})\cap \tfrac{RD}{\ell_F}\] given by $x\mapsto \tfrac{x}{\ell_F}$, so it suffices to estimate the number of elements in the latter set. Since $m_F\leq 1$ it follows that $\left(\mathcal{L}\cap \left(\tfrac{RD}{\ell_F}\times\tfrac{m_F\mathcal{W}_\mathcal{A}}{|\overline{\ell_F}|} \right)\right)\smpt{0}\subset \left(\mathcal{L}\cap \left(\tfrac{RD}{\ell_F}\times\tfrac{\mathcal{W}_\mathcal{A}}{|\overline{\ell_F}|} \right)\right)\smpt{0}$. Fix real numbers $m_1,m_2>1$ so that $D\subset [-m_1,m_1]^2=:B_1$ and $\mathcal{W}_\mathcal{A}\subset [-m_2,m_2]^2=:B_2$.
		
		Fix a number $c$ so that $c'<c$ implies $(\mathcal{L}\cap(\lambda D\times c'\mathcal{W}_\mathcal{A}))\smpt{0}=\emptyset$. This can be done, for otherwise $(\mathcal{L}\cap(\lambda D\times c'\mathcal{W}_\mathcal{A}))\smpt{0}$ would be non-empty for each $c'>0$, hence $\mathcal{L}\cap (\lambda D\times\mathcal{W}_\mathcal{A})$ would contain infinitely many points, contradiction, since $\mathcal{L}$ is a lattice and $\lambda D\times\mathcal{W}_\mathcal{A}$ is bounded.
		
		Suppose first that $\frac{R}{|\ell_F\overline{\ell_F}|}<c$. Scale $\ell_F$ by units so that $1\leq \tfrac{R}{\ell_F}<\lambda$ which gives $\tfrac{1}{|\overline{\ell_F}|}<c$. Hence
		$\left(\mathcal{L}\cap\left(\tfrac{RD}{\ell_F}\times \tfrac{\mathcal{W}_\mathcal{A}}{|\overline{\ell_F}|}\right)\right)\smpt{0}\subset \left(\mathcal{L}\cap\lambda D\times\left( \tfrac{\mathcal{W}_\mathcal{A}}{|\overline{\ell_F}|}\right)\right)\smpt{0}=\emptyset$
		and therefore $\#\left(\left(\mathcal{L}\cap\left(\tfrac{RD}{\ell_F}\times \tfrac{\mathcal{W}_\mathcal{A}}{|\overline{\ell_F}|}\right)\right)\smpt{0}\right)=0$.
	
		Suppose now that $\frac{R}{|\ell_F\overline{\ell_F}|}\geq c$. Scale $\ell_F$ so that $\sqrt{c}\leq \tfrac{R}{\ell_F}<\lambda \sqrt{c}$. This implies that $\tfrac{1}{|\overline{\ell_F}|}\geq \tfrac{1}{\lambda}\sqrt{c}>\sqrt{c}$. Thus, $[0,\sqrt{c}]^4\subset \tfrac{RB_D}{\ell_F}\times \tfrac{B_\mathcal{W}}{|\overline{\ell_F}|}=:B$. From \Cref{lemLatticeBoxBound} we get a constant $L$ only depending on $\mathcal{L}$, $\sqrt{c}$ such that $
		\#(B\cap \mathcal{L})\leq L\vol(B)=L\cdot 16m_1^2m_2^2\frac{R^2}{N(\ell_F)^2}.$
		Now, since $\tfrac{RD}{\ell_F}\times \tfrac{\mathcal{W}_\mathcal{A}}{|\overline{\ell_F}|}\subset B$ we get that
		\[\#\left(\left(\mathcal{L}\cap\tfrac{RD}{\ell_F}\times \left(\tfrac{\mathcal{W}}{|\overline{\ell_F}|}\right)\right)\smpt{0}\right)\leq\frac{\widetilde{L}R^2}{N(\ell_F)^2}\]
		with $\widetilde{L}:=16m_1^2m_2^2L$.
	\end{proof}
	
	\begin{prop}
	\label{propInterchange}
	The equality
	\[\lim_{R\to\infty}\sum_{\substack{F\subset C\\\#F<\infty}}\frac{(-1)^{\#F}\#\left((\mathcal{P}(m_F\mathcal{W}_\mathcal{A},\mathcal{L}_F)\smpt{0})\cap RD\right)}{\vol(RD)}=\sum_{\substack{F\subset C\\\#F<\infty}}\frac{(-1)^{\#F}m_F^2(1+\sqrt{2})}{2N(\ell_F)^2}\]
	holds for all Jordan measurable $D\subset \bb{R}^2$.
	\end{prop}
	
	\begin{proof}
	For a finite $F\subset C$ let $N(R,F)=\#\left((\mathcal{P}(m_F\mathcal{W}_\mathcal{A},\mathcal{L}_F)\smpt{0})\cap RD\right)$. We know that $\lim\limits_{R\to\infty}\frac{N(R,F)}{\vol(RD)}=\frac{m_F^2(1+\sqrt{2})}{2N(\ell_F)^2}$ so
	\begin{equation}
	\label{eqnSwitchLimitsJustification}
	\lim\limits_{R\to\infty}\sum_{\substack{F\subset C\\\#F<\infty}}\frac{(-1)^{\#F}N(R,F)}{\vol(RD)}=\sum_{\substack{F\subset C\\\#F<\infty}}\lim\limits_{R\to\infty}\frac{(-1)^{\#F}N(R,F)}{\vol(RD)}
	\end{equation}
	must be justified. In view of \Cref{lem-Estimate1}
	\[\sum_{\substack{F\subset C\\\#F<\infty}}\left|\frac{(-1)^{\#F}N(R,F)}{\vol(RD)}\right|\leq \frac{\widetilde{L}}{\vol(D)}\sum_{\substack{F\subset C\\\#F<\infty}}\frac{1}{N(\ell_F)^2}\]
	and we note that 
	\[\sum_{\substack{F\subset C\\\#F<\infty}}\frac{1}{N(\ell_F)^2}=\sum_{\substack{F\subset C\\\#F<\infty\\\lambda\in F}}\frac{1}{N(\ell_F)^2}+\sum_{\substack{F\subset C\\\#F<\infty\\\lambda\notin F}}\frac{1}{N(\ell_F)^2}\leq 2\sum_{I\triangleleft \mathcal{O}_K}\frac{1}{N(I)^2},\]
	hence the sums of both sides of \eqref{eqnSwitchLimitsJustification} are absolutely convergent.
	
	Fix $\Delta>0$. We claim that there is only a finite number of non-empty $F\subset C$, $\#F<\infty$, such that $|N(\ell_F)|<\Delta$. Given such $F$ let $\ell_F=\prod_{c\in F}c>1$. Also, since $|\overline{\pi}|>1$ for all $\pi\in \bb{P}$ we have $|\overline{\ell_F}|\geq |\overline{\lambda}|$. Hence, $|N(\ell_F)|=\ell_F|\overline{\ell_F}|\leq \Delta$ implies $\ell_F\leq \frac{\Delta}{|\overline{\ell_F}|}\leq \lambda \Delta$ and $|\overline{\ell_F}|\leq \frac{\Delta}{\ell_F}\leq \Delta<\lambda\Delta$ so $(\ell_F,\overline{\ell_F})\in \{(x,\overline{x})\mid x\in\mathcal{O}_K\}\cap \lambda[-\Delta,\Delta]^2$ which is a finite set, thus elements of $F$ can only contain prime factors that occur as factors in the components of elements in this finite set, giving only finitely many possibilities for $F$.
	
	It follows that
	\begin{align*}
	&\left|\lim\limits_{R\to\infty}\sum_{\substack{F\subset C\\\#F<\infty}}\frac{(-1)^{\#F}N(R,F)}{\vol(RD)}-\sum_{\substack{F\subset C\\\#F<\infty}}\frac{m_F^2(1+\sqrt{2})}{2N(\ell_F)^2}\right|\\
	&\leq \left(\frac{\widetilde{L}}{\vol(D)}+\frac{1+\sqrt{2}}{2}\right)\sum_{\substack{F\subset C\\\#F<\infty\\|N(\ell_F)|\geq \Delta}}\frac{1}{N(\ell_F)^2}
	\end{align*}
	where the right hand side tends to $0$ as $\Delta\to\infty$ since $\sum\limits_{F\subset C,\#F<\infty,|N(\ell_F)|\geq \Delta}\frac{1}{N(\ell_F)^2}$ is the tail of an absolutely convergent sum, hence \eqref{eqnSwitchLimitsJustification} has been justified.
	\end{proof}
	
	From \Cref{propInterchange} it follows that $\theta(\widehat{\mathcal{B}})=\sum_{\substack{F\subset C\\\#F<\infty}}\frac{(-1)^{\#F}m_F^2(1+\sqrt{2})}{2N(\ell_F)^2}$, and it will now be shown that the right hand side is equal to $1/\zeta_K(2)$. Define the function $\omega:\mathcal{O}_K\longrightarrow\bb{C}$, $\omega(x)=\#\{\pi\in\bb{P}\mid x/\pi\in\mathcal{O}_K\}$, so that $\omega(x)$ is the number of non-associated prime divisors of $x$. Given $I\triangleleft \mathcal{O}_K$, let $\omega(I)=\omega(x)$ for any generator $x$ of $I$ and define a Möbius function on the ideals of $\mathcal{O}_K$ by 
	\[\mu(I)=\begin{cases}0 & \text{ if }\exists \pi\in\bb{P}\text{ such that } I\subset \pi^2\mathcal{O}_K,\\
	(-1)^{\omega(I)} &\text{ otherwise.}
	\end{cases}\]
	One verifies that $\mu(I_1I_2)=\mu(I_1)\mu(I_2)$ for relatively prime ideals $I_1,I_2$. The function $\zeta_K$ can be expressed as an Euler product for $s$ with $\mathrm{Re}(s)>1$ as
	\[\zeta_K(s)=\prod_{P\triangleleft \mathcal{O}_K,\,P\,\mathrm{prime}}\frac{1}{1-N(P)^{-s}}\]
	and in analogy with the reciprocal formula for Riemann's zeta function we have
	\begin{equation}
	\label{eqnRecDedekind}
	\frac{1}{\zeta_K(s)}=\sum_{I\triangleleft \mathcal{O}_K}\frac{\mu(I)}{N(I)^s}.
	\end{equation}
	
	\begin{thm}
	\label{thmAsDensA'}
	The density of visible points of $\mathcal{B}$ is given by
	\[\theta(\mathcal{B})=\frac{1}{\zeta_K(2)}=\frac{48\sqrt{2}}{\pi^4}.\]
	\end{thm}
	
	\begin{proof}
	By \Cref{propInterchange} we have \[\theta(\mathcal{B})=\sum_{\substack{F\subset C\\\#F<\infty}}\frac{(-1)^{\#F}m_F^2(1+\sqrt{2})}{2N(\ell_F)^2}.\]
	Splitting the sum into two depending on whether $\lambda \in F$ or not, and using that $m_F$ is $1$ unless $\lambda\in F$, in which case $m_F=|\overline{\lambda}|=\sqrt{2}-1$, we get
	\begin{align*}
	\theta(\mathcal{B})&=\sum_{\substack{F\subset C\\\#F<\infty\\\lambda\notin F}}\frac{(-1)^{\#F}(1+\sqrt{2})}{2N(\ell_F)^2}+\sum_{\substack{F\subset C\\\#F<\infty\\\lambda\in F}}\frac{(-1)^{\#F}|\overline{\lambda}|^2(1+\sqrt{2})}{2N(\ell_F)^2}\\
	&=\frac{(1-|\overline{\lambda}|^2)(1+\sqrt{2})}{2}\sum_{I\triangleleft\mathcal{O}_K}\frac{\mu(I)}{N(I)^2}=\frac{1}{\zeta_K(2)},
	\end{align*}
	last equality by \eqref{eqnRecDedekind}. From \cite[Theorem 4.2]{washington1997introduction} one can calculate $\zeta_K(-1)=\frac{1}{12}$ and by the functional equation for Dedekind's zeta function (cf. e.g. \cite[p. 34]{washington1997introduction}) one finds that $\zeta_K(2)=\frac{\pi^4}{48\sqrt{2}}$ which proves the claim.
	\end{proof}
	
	\subsection{The density of visible points of $\mathcal{A}$}
	\label{subsecDensAB}
	
	Observe that $\mathcal{A}'=\sqrt{2}\mathcal{A}\subset \mathcal{B}$. It is now shown that $C$ is also an occluding set for $\mathcal{A}'$.
	
	\begin{prop}
		\label{propOcclSet2}
		For each $x\in\mathcal{A}'\setminus\widehat{\mathcal{A}'}$ there is $c\in C$ such that $x/c\in \mathcal{A}'$.
	\end{prop}
	
	\begin{proof}	
	Since $\mathcal{A}'\subset \mathcal{B}$ we have $\mathcal{A}'\setminus\widehat{\mathcal{A}'}\subset \mathcal{B}\setminus\widehat{\mathcal{B}}$ and so for each $x\in \mathcal{A}'\setminus\widehat{\mathcal{A}'}$ there exists $c\in C$ such that $x/c\in \mathcal{B}$. If $c\neq \sqrt{2}$ then $\sqrt{2}\mid \frac{x_1-x_2}{c}$ so $x/c\in\mathcal{A}'$.
		
	Take now $x\in \mathcal{A}'\setminus\widehat{\mathcal{A}'}$ such that for all $c\in C\smpt{\sqrt{2}}$ we have $x/c\notin \mathcal{B}$. Then $x/\sqrt{2}\in\mathcal{B}$, hence $\gcd(x_1,x_2)=\sqrt{2}^n$ for some $n\geq1$. Since $x\in \mathcal{A}'\setminus\widehat{\mathcal{A}'}$ there is $c\in \bb{Q}(\sqrt{2})\cap\bb{R}_{>1}$ such that $x/c\in \mathcal{A}'$. Writing $c=a/b$ with $\gcd(a,b)=1$, the only possible $\pi\in\bb{P}$ with $\pi\mid a$ is $\pi=\sqrt{2}$. If $\sqrt{2}\mid a$, then it follows that $x/\sqrt{2}\in\mathcal{A}'$.
		
	It remains to check the case where $a$ is a unit, i.e. $c=\frac{\lambda^n}{\prod_{p\in \bb{P}}\pi^{m(\pi)}}$ for some $m:\bb{P}\longrightarrow\bb{Z}_{\geq 0}$ with finite support. The facts that $c>1$ and $\pi>1$ for all $\pi\in\bb{P}$ imply $n>0$. We have $x/\lambda\notin\mathcal{B}$, hence $\overline{x}\notin |\overline{\lambda}| \mathcal{W}_\mathcal{A}$. Since $x/c\in\sqrt{2}\mathcal{A}$ it follows that $\overline{x}\in |\overline{c}|\mathcal{W}_\mathcal{A}$ and hence $|\overline{c}|>|\overline{\lambda}|$. However
	\[|\overline{c}|=\frac{|\overline{\lambda}|^n}{\prod_{\pi\in \bb{P}}|\overline{\pi}|^{k(\pi)}}\leq |\overline{\lambda}|^n\leq |\overline{\lambda}|,\]
	contradiction.
	\end{proof}
	
	\begin{thm}
	We have $\theta(\widehat{\mathcal{A}'})=\frac{1}{2\zeta_K(2)}$, hence $\theta(\widehat{\mathcal{A}})=\frac{1}{\zeta_K(2)}$. 
	\end{thm}
	
	\begin{proof}
	Propositions \ref{propInclExcl}, \ref{propOcclSet2} imply
	\begin{equation}
	\label{eqnB1}
	\frac{\#(\widehat{\mathcal{A}'}\cap RD)}{\vol(RD)}=\sum_{\substack{F\subset C\\ \#F<\infty}}\frac{(-1)^{\#F}\#\left(\left(\mathcal{A}_*'\cap \bigcap_{c\in F}c\mathcal{A}_*'\right)\cap RD\right)}{\vol(RD)}
	\end{equation}
	and it is straight-forward to verify that $\mathcal{A}_*'\cap \bigcap_{c\in F}c\mathcal{A}_*'=\mathcal{P}(m_F\mathcal{W}_\mathcal{A},\widetilde{\mathcal{L}}_F)\smpt{0}$ with $\widetilde{\mathcal{L}}_F=\{(\ell_Fx,\overline{\ell_Fx})\mid x\in \mathcal{O}_K^2,(x_1-x_2)/\sqrt{2}\in\mathcal{O}_K\}$ a sublattice of $\mathcal{\mathcal{L}}_F$ of index $2$. Hence, by \cite[Prop. 3.2]{marklof2014free}, when letting $R\to\infty$ inside the sum \eqref{eqnB1} one obtains
	\[\sum_{\substack{F\subset C\\ \#F<\infty}}\frac{(-1)^{\#F}\vol(m_F\mathcal{W}_\mathcal{A})}{16N(\ell_F)},\]
	whence $\theta(\widehat{\sqrt{2}\mathcal{A}})=\frac{1}{2\zeta_K(2)}$ follows by \Cref{propInterchange} and \Cref{thmAsDensA'}, and the other result is immediate as $\sqrt{2}\mathcal{A}=\mathcal{A}'$.
	\end{proof}
	
	\begin{rem}
	The data of Table 2 of \cite{baake2014radial} shows that $\#(\widehat{\mathcal{A}}\cap RD)/\#(\mathcal{A}\cap RD)\approx 0.577$ for a particular $D$ and fairly large $R$. This agrees with our results, since 
	\[\kappa_\mathcal{A}=\lim_{R\to\infty}\frac{\#(\widehat{\mathcal{A}}\cap RD)}{\#(\mathcal{A}\cap RD)}=\frac{\theta(\widehat{\mathcal{A}})}{\theta(\mathcal{A})}=\frac{\tfrac{1}{\zeta_K(2)}}{\tfrac{2\vol(\mathcal{W}_\mathcal{A})}{16}}=\frac{2(\sqrt{2}-1)}{\zeta_K(2)}=0.5773\ldots\]
	\end{rem}
	
	\bibliographystyle{siam}
	\bibliography{bibl}	

\begin{thebibliography}{1}

\bibitem{baake2014radial}
{\sc M.~Baake, F.~G{\"o}tze, C.~Huck, and T.~Jakobi}, {\em Radial spacing
  distributions from planar point sets.}, Acta crystallographica. Section A,
  Foundations and advances, 70 (2014), pp.~472--482.

\bibitem{baake2013aperiodic}
{\sc M.~Baake and U.~Grimm}, {\em Aperiodic Order}, vol.~1, Cambridge
  University Press, 2013.

\bibitem{baake2000diffraction}
{\sc M.~Baake, R.~V. Moody, and P.~A. Pleasants}, {\em Diffraction from visible
  lattice points and kth power free integers}, Discrete Mathematics, 221
  (2000), pp.~3--42.

\bibitem{marklof2014free}
{\sc J.~Marklof and A.~Str{\"o}mbergsson}, {\em Free path lengths in
  quasicrystals}, Communications in mathematical physics, 330 (2014),
  pp.~723--755.

\bibitem{marklof2014visibility}
\leavevmode\vrule height 2pt depth -1.6pt width 23pt, {\em Visibility and
  directions in quasicrystals}, International mathematics research notices,
  2015 (2014), pp.~6588--6617.

\bibitem{nymann1972probability}
{\sc J.~Nymann}, {\em On the probability that k positive integers are
  relatively prime}, Journal of number theory, 4 (1972), pp.~469--473.

\bibitem{singppt1}
{\sc B.~Sing}, {\em Visible {A}mmann-{B}eenker points}.
\newblock \url{http://www.bb-math.com/bernd/pub/bcc.pdf}, 2007.

\bibitem{washington1997introduction}
{\sc L.~C. Washington}, {\em Introduction to cyclotomic fields}, vol.~83,
  Springer Science \& Business Media, 1997.

\end{thebibliography}
	
\end{document}